%% file: reducibles.tex
\documentclass{article}

\usepackage {amsmath, amssymb, amscd, amsthm, mathabx, mathrsfs, hyperref, enumerate, url,graphicx}
\usepackage{caption,subcaption,pinlabel}
\usepackage[text={6in,8.5in}]{geometry}
\usepackage{color}
\usepackage[all]{xy}

\newtheorem {theorem}{Theorem}[section]
\newtheorem {lemma}[theorem]{Lemma}
\newtheorem {proposition}[theorem]{Proposition}
\newtheorem {corollary}[theorem]{Corollary}
\newtheorem {conjecture}[theorem]{Conjecture}

\newtheorem {remark}[theorem]{Remark}

\numberwithin{equation}{section}

\DeclareMathOperator{\rank}{rank}
\DeclareMathOperator{\Hom}{Hom}
\newcommand{\F}{\mathbb{F}}
\newcommand{\Z}{\mathbb{Z}}
\newcommand{\R}{\mathbb{R}}
\newcommand{\HFhat}{\widehat{HF}}

\newcommand {\s} {\mathfrak{s}}
\newcommand{\A}{\widehat{\mathcal{A}}}
\newcommand{\B}{\widehat{\mathcal{B}}}
\newcommand{\h}{\widehat{h}}
\renewcommand{\v}{\widehat{v}}

\newcommand{\cA}{\mathcal{A}}
\newcommand{\cB}{\mathcal{B}}
\newcommand{\cT}{\mathcal{T}}
\newcommand{\cX}{\mathcal{X}}
\newcommand{\HFKhat}{\widehat{HFK}}
\newcommand{\mfv}{\mathfrak{v}}
\newcommand{\mfh}{\mathfrak{h}}
\newcommand{\mfvhat}{\widehat{\mathfrak{v}}}
\newcommand{\mfhhat}{\widehat{\mathfrak{h}}}
\newcommand{\Ahat}{\widehat{A}}
\newcommand{\Bhat}{\widehat{B}}
\newcommand{\HFred}{HF^+_\textup{red}}

\newcommand{\spinc}{\mathfrak{s}}
\newcommand{\tower}{\mathcal{T}^+}
\DeclareMathOperator{\gr}{gr}

\DeclareMathOperator{\coker}{coker}
\newcommand{\grbottom}{\gr^{\textup{bot}}}
\newcommand{\grtop}{\gr^{\textup{top}}}
\newcommand{\HFdiamond}{\widecheck{HF}}

\makeatletter
\let\@fnsymbol\@arabic
\makeatother

\title{Reducible surgeries and Heegaard Floer homology}
\author{Jennifer Hom\footnote{Department of Mathematics, Columbia University, New York, NY 10027, USA, \url{hom@math.columbia.edu}} \and Tye Lidman\footnote{Department of Mathematics, The University of Texas, Austin, TX 78712, USA, \url{tlid@math.utexas.edu}} \and Nicholas Zufelt\footnote{Department of Mathematics, The University of Texas, Austin, TX 78712, USA, \url{nzufelt@math.utexas.edu}}}
\date{} 

\begin{document}
\maketitle

\abstract{In this paper, we use Heegaard Floer homology to study reducible surgeries.  In particular, suppose $K$ is a non-cable knot in $S^3$ with a positive $L$-space surgery.  If $p$-surgery on $K$ is reducible, we show that $p=2g(K)-1$.  This implies that any knot with an $L$-space surgery has at most one reducible surgery, a fact that we show additionally for any knot of genus at most two.}

\input{introduction}

\input{review}

\input{lspace}

\input{doublereducibles}

\bibliographystyle{amsalpha}
\bibliography{bibliography}
\end{document}

%% file: introduction.tex
\section{Introduction}\label{sec:introduction}

Let $K$ be a knot in $S^3$, and denote by $S^3_{\frac{p}{q}}(K)$ the result of $\frac{p}{q}$-Dehn surgery on $K$.  Recall that a 2-sphere in a 3-manifold is \textit{essential} if it is not the boundary of any ball in the manifold. In this paper, we will be concerned about the case when $S^3_{\frac{p}{q}}(K)$ is a \textit{reducible} manifold, that is, a manifold which contains an essential sphere.  Such a slope $\frac{p}{q}$ will be called a \textit{reducing slope}.  By \cite{pap:gabai}, we may assume that in this situation $S^3_{\frac{p}{q}}(K)$ will decompose as a connected sum and that $p\neq 0$.  An example of this occurs when the knot $K$ is a cable.  More specifically, if $K$ is the $(p,q)$-cable of a knot $K'$, then $S^3_{\frac{pq}{1}}(K)\cong L(q,p)\# S^3_{\frac{p}{q}}(K')$, where $L(q,p)$ is the lens space given by $\frac{q}{p}$-surgery on the unknot.  In fact, the Cabling Conjecture asserts that this is the only example:

\begin{conjecture}[Cabling Conjecture, Gonzalez-Acu\~na -- Short \cite{pap:gonz-short}] 
If $K$ is a knot in $S^3$ which has a reducible surgery, then $K$ is a cable and the reducing slope is given by the cabling annulus.  
\end{conjecture}

The reducing slope given by the cabling annulus is precisely the example mentioned above.  Here, torus knots are considered cables of the unknot.  The Cabling Conjecture is true for many classes of knots, including torus knots \cite{pap:moser}, satellite knots \cite{pap:scharlemann}, and alternating knots \cite{pap:menasco-this}.  In particular it suffices to consider hyperbolic knots.  

Note that for cabled knots, the reducing slope is an integer and one of the connected summands is a lens space.  These conditions are known to be true for any reducible Dehn surgery (see \cite{pap:gordon-luecke-integral} and \cite{pap:gordon-luecke-complement},  respectively).  In particular, neither $S^3_{\pm 1}(K)$ nor $S^3_0(K)$ is reducible.  Thus, assume $1<|\frac{p}{q}|=|p|$.  Additionally, it is known from \cite{pap:gordon-luecke-distance} that the geometric intersection number between any two distinct reducing slopes for a knot is 1.  Since these slopes are integral, it follows that any knot in $S^3$ has at most two reducible surgeries, which would necessarily be consecutive integers.  In addition, it is shown in \cite{pap:matignon-sayari} that for non-cable knots, any reducing slope $p$ satisfies the bound $|p| \leq 2g-1$, where $g$ denotes the genus of $K$.  This bound, along with the other results previously mentioned, severely restricts the possible reducing slopes of a non-cable knot: it must be an integer $p$ in the range
\begin{equation}\label{genusbound} 1 < |p| \leq 2g-1.\end{equation}

Heegaard Floer homology, defined in \cite{hfinvariance}, has proved very useful in answering questions involving Dehn surgery.  Recall that to each closed, oriented 3-manifold $Y$, Ozsv\'ath and Szab\'o associate to $Y$ a finitely generated abelian group $\HFhat(Y)$, which splits over Spin$^c$ structures:
\[
\HFhat(Y) = \bigoplus_{\s \in \text{Spin}^c(Y)} \HFhat(Y,\s).
\]
Throughout this paper all Heegaard Floer chain complexes will be computed with $\F = \Z/2\Z$ coefficients, and we therefore omit them from the notation.  We will also work with the related invariant $HF^+(Y)$, a module over the ring $\F[U]$, where $U$ is a formal variable.

Heegaard Floer homology has been especially fruitful when studying surgeries which yield lens spaces.  This is facilitated by the simplicity of their Heegaard Floer groups.  In particular, a lens space $Y$ satisfies $\dim \HFhat(Y) = |H^2(Y;\Z)|$.  More generally, any rational homology sphere satisfies the inequality $\dim\HFhat(Y,\s)\geq 1$ for each $\s$; an {\em $L$-space} is a rational homology sphere where equality is satisfied for all $\s$.  A knot in $S^3$ with a positive $L$-space surgery is called an {\em $L$-space knot}.  Obviously, a knot with an $L$-space surgery is either an $L$-space knot or its mirror is.  Admitting such a surgery imposes strong restrictions on the knot, such as being fibered \cite{pap:ghigginifibered, yihfkfibered}.  

There is a sort of complementary result to \eqref{genusbound} given in \cite{pap:kron-mrow-ozsv-szab} for $L$-space surgeries: if $K$ has a positive $L$-space surgery, then its slope satisfies $\frac{p}{q} \geq 2g-1$.  Hence if a reducible $L$-space is obtained by surgery on a hyperbolic knot in $S^3$, the slope must be $\pm (2g-1)$.  It is shown in \cite{pap:greene} that reducible $L$-spaces that bound a so-called sharp 4-manifold cannot arise from Dehn surgery on a knot in $S^3$.  It also follows from \cite{pap:greene}, combined with \cite{pap:boyer-zhang2}, that a hyperbolic knot in $S^3$ cannot admit both a lens space surgery and a reducible surgery.  

The goal of this paper is to show that Heegaard Floer homology is a potentially useful tool to study the Cabling Conjecture.  One of the key observations is that Heegaard Floer homology satisfies a K\"unneth formula \cite[Theorem 1.5]{hfpa}, in the sense that if $Y\cong Y_1\#Y_2$, then   
\begin{equation}\label{kunneth}\HFhat(Y,\mathfrak{s}_1\#\mathfrak{s}_2)\cong \HFhat(Y_1,\mathfrak{s}_1)\otimes\HFhat(Y_2,\mathfrak{s}_2),\end{equation}
since we are working with $\F$-coefficients.  This property and its analogue for $HF^+$, combined with the fact that lens spaces are $L$-spaces, gives a pattern to the Heegaard Floer homology groups of a manifold with a lens space summand (see Lemma \ref{lem:tperiodic}).  We will then compare this to the structure of the Heegaard Floer homology of a manifold obtained by surgery on a knot using the mapping cone formula of \cite{hfkz}.  

As a warm-up with the techniques, we will give a proof of a classical result.  

\begin{theorem}[Boyer-Zhang \cite{pap:boyer-zhang}]\label{thm:genusone}
Genus 1 knots satisfy the Cabling Conjecture.  
\end{theorem}

Note that Theorem \ref{thm:genusone} additionally follows from \eqref{genusbound}.  While our proof does not rely on the upper bound in \eqref{genusbound}, it will play an important role in the proof of our main theorem, which determines the possible reducing slope of $L$-space knots.

\begin{theorem}\label{thm:maintheorem}
Let $K$ be a hyperbolic $L$-space knot.  If $p$ is a reducing slope, then $p = 2g-1$.  
\end{theorem}

We also have the following corollaries of Theorem \ref{thm:maintheorem}.

\begin{corollary}\label{cor:lstworeducibles}
If $K$ is an $L$-space knot, then $K$ cannot admit two reducible surgeries.  
\end{corollary}

\begin{corollary}\label{cor:puncturedrp2}
Let $K$ be a hyperbolic knot in $S^3$ with an $L$-space surgery.  Then the exterior of $K$ contains no properly embedded punctured projective planes. 
\end{corollary}
\begin{proof}
Suppose $P$ is a properly embedded punctured projective plane in the exterior of $K$.  Then there exists a slope $p$ given by $\partial P$ in which the resulting surgered manifold $M$ contains a projective plane $\hat{P}$.  Let $N$ denote the closure of a tubular neighborhood of $\hat{P}$ in $M$, so that $N$ is a once-punctured $\R P^3$.  By the $\R P^3$ theorem \cite{pap:kron-mrow-ozsv-szab}, $M\not\cong \R P^3$, and hence $M\setminus \text{int}(N)$ is not a ball.  Thus $M$ is reducible, so by Theorem \ref{thm:maintheorem} the surgery slope is an odd integer.  However, $M\cong\R P^3 \# M'$ for some $M'$, and hence $|H_1(M)| = p$ is even, a contradiction.
\end{proof}

In light of Corollary \ref{cor:lstworeducibles}, one might ask what can be said without the hypothesis of having an $L$-space surgery.  We have the following theorem. 

\begin{theorem}\label{thm:tworeducibles}
If $K$ is a knot in $S^3$ of genus at most two, then $K$ does not admit two reducible surgeries.  
\end{theorem}


\section*{Acknowledgments}
We would like to thank Cameron Gordon for encouragement and Matt Hedden for a helpful discussion.  The first author was partially supported by NSF grant DMS-1307879. The second and third authors were partially supported by NSF RTG grant DMS-0636643.

%% file: review.tex
\section{Heegaard Floer homology preliminaries}\label{sec:review}
We will assume the reader is familiar with all flavors of Heegaard Floer homology for a $3$-manifold $Y$ (defined in \cite{hfinvariance}), as well as the $\Z \oplus \Z$-filtered knot Floer chain complexes $CFK^\infty(Y,K)$ for a $3$-manifold, null-homologous knot pair $(Y,K)$ (defined in \cite{knotinvariants} and \cite{rasthesis}).  

Let $K$ be a nontrivial knot in $S^3$, in which case we simply write $CFK^\infty(K)$ to denote $CFK^\infty(S^3, K)$.  We are interested in computing $\HFhat(S^3_p(K))$ and $HF^+(S^3_p(K))$ for $p$ an integer.  In order to do this, we use the mapping cone formula given in \cite{hfkz}.  We now briefly sketch the construction to establish notation. 

\subsection{Review of the mapping cone formula}
For a subset $X$ of $\Z \oplus \Z$, let $CX$ be the subgroup of $CFK^\infty(K)$ generated by those elements with filtration level $(i,j)\in X$.  The set $CX$ will have an induced chain complex structure if it may be obtained by passing from $CFK^\infty(K)$ to successive subcomplexes or quotient complexes.
 
For $s\in \Z$ define complexes
\begin{align*}
\A_s &= C\{\text{max}\{i,j-s\} = 0\},\\
\B_s &= C\{i = 0\}
\end{align*}
and
\begin{align*}
\cA^+_s &= C\{\text{max}\{i,j-s\} \geq 0\},\\
\cB^+_s &= C\{i \geq 0\}.
\end{align*}
Throughout this section, we will use $^\circ$ to denote either $\ \widehat{} \ $ or $^+$. Note that $\cB^\circ\cong CF^\circ(S^3)$.  Fix a nonzero integer $p$, and consider performing $p$-surgery on $K$.  There are two canonical chain maps $\mfv^\circ_s:\cA^\circ_s \to \cB^\circ_s$ and $\mfh^\circ_s:\cA^\circ_s \to \cB^\circ_{s+p}$, where $\mfvhat_s$ (respectively $\mfv^+_s$) is given by projection to $C\{ i=0\}$ (respectively projection to $C\{ i \geq 0\}$), and $\mfhhat_s$ (respectively $\mfh^+_s$) is a composition of three chain maps: the first is projection to $C\{j= s\}$ (respectively $C\{ j \geq s\}$), the second is a chain homotopy equivalence which shifts vertically to $C\{j=0\}$ (respectively $C\{ j \geq 0\}$), and the final is a chain homotopy equivalence to $C\{i=0\}$ (respectively $C\{i \geq 0\}$) induced by Heegaard moves.  

For $s\in \Z$, define chain maps 
\[ D^\circ_{p,s}:\mathop{\bigoplus}_{t\equiv s (\text{mod } p)}\cA^\circ_t\longrightarrow\mathop{\bigoplus}_{t\equiv s (\text{mod } p)}\cB^\circ_t, \]
where for $x \in \cA^\circ_t$, 
\[ D^\circ_{p,s}(x)=\mfv^\circ_t(x) + \mfh^\circ_t(x). \]
It is clear that $D^\circ_{p,s}=D^\circ_{p,s'}$ for $s\equiv s' \pmod p$.  Define $\cX^\circ_{p,s} =\text{Cone}(D^\circ_{p,s})$, where Cone denotes the (homological) mapping cone.  Define 
\[ \cX^\circ_p=\mathop{\bigoplus}_{0\leq  s < p} \cX^\circ_{p,s}. \]

For integral surgeries, there is a canonical identification of $\text{Spin}^c(S^3_p(K))$ with $\Z/p\Z$, described explicitly in \cite[Subsection 2.4]{hfkz}.  We will assume throughout the remainder of this paper that this identification has been fixed for any such surgered manifold.  
For $s\in \Z$, we will use $[s]$ to denote the $\text{Spin}^c$ structure corresponding to the mod $p$ residue class of $s$.  The identification is $\Z/p\Z$-equivariant in the sense that $[s + 1] = [s] + PD[\mu]$, where $\mu$ represents the meridian of $K$.  Under this identification, the complexes $\cA^\circ_s$ are quasi-isomorphic to $CF^\circ(S^3_m(K),[s])$ for $m$ sufficiently larger than $s$ \cite[Theorem 4.4]{knotinvariants}.  More generally, we have the following.

\begin{theorem} [Ozsv\'ath-Szab\'o \cite{hfkz}]\label{thm:mcf} Let $\cX^\circ_{p,s}$ be as defined above.  Then there is an isomorphism
$$H_*(\cX^\circ_{p,s}) \cong HF^\circ(S^3_p(K),[s]).$$
\end{theorem}

\subsection{Properties of the mapping cone formula}
There is a standard diagram associated to $\cX^\circ_p$, to which we will refer as the \textit{mapping cone diagram}.  An example of this is given in Figure~\ref{fig:mappingconediagram}.  We will not distinguish between the diagram and the corresponding chain complex.  
\begin{figure}[h]
\includegraphics[width=.99\textwidth]{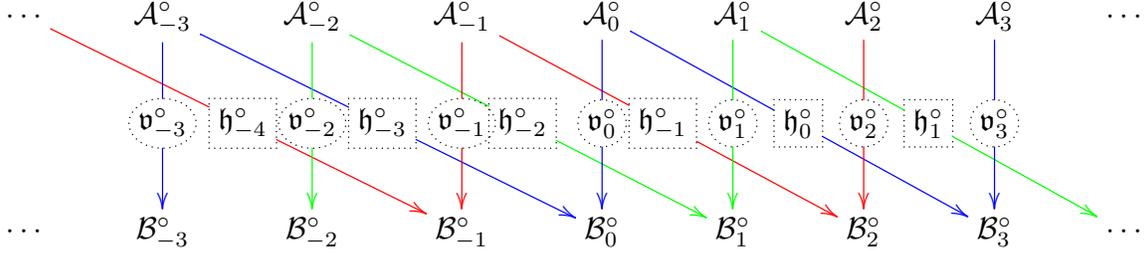}
\caption{The mapping cone diagram for $p = 3$.  Each color in the diagram corresponds to $\cX^\circ_{p,s}$ for some $s$.}\label{fig:mappingconediagram}
\end{figure}

To simplify our study of $\cX^\circ_{p,s}$ we pass to a smaller, quasi-isomorphic chain complex. Let $g$ denote the genus of $K$.  Recall that
\[ g = \max \{ t \in \Z \mid \HFKhat(K, t) \neq 0 \}, \]
where $\HFKhat(K, t)$ denotes the summand of $\HFKhat$ in Alexander grading $t$ \cite[Theorem 1.2]{genusdetection}. As in the proof of \cite[Proposition 9.6]{rationalsurgeries}, this implies that the map
\[ \mfv^\circ_s: \cA^\circ_s \rightarrow \cB^\circ_s \]
induces an isomorphism on homology for all $s \geq g$. Indeed, the kernel of $\mfvhat_s$ (respectively $\mfv^+_s$) is the subcomplex $C\{ i <0, j=s\}$ (respectively $C\{ i<0, j \geq s\}$), which is acyclic when $s \geq g$. Furthermore, the image of $\mfv^\circ_s$ in $\cB^\circ_s$ has acyclic quotient. Similarly, the map
\[ \mfh^\circ_s: \cA^\circ_s \rightarrow \cB^\circ_{s+p} \]
induces an isomorphism on homology for all $s \leq -g$

These facts imply the following proposition.

\begin{proposition}\label{prop:truncatedmcf}  Define the \textbf{truncated mapping cone diagram}, denoted $X_p^\circ$, to be the following:\begin{itemize}
\item The objects in the diagram are the homology groups $A^\circ_s=H_*(\cA^\circ_s)$ for $1-g \leq s \leq \max\{g-1, 1-g+p\}$, as well as $B^\circ_s = H_*(\cB^\circ_s)$ for $1-g + p\leq s \leq g-1$.
\item The maps in the diagram are the induced maps $v^\circ_s=(\mfv^\circ_s)_*$ which have both their domain and codomain in the object set, and similarly for the induced maps $h^\circ_s=(\mfh^\circ_s)_*$.
\end{itemize}
The complex $\cX^\circ_{p,s}$, and hence the complex $CF^\circ(S^3_p(K),[s])$, is quasi-isomorphic to the complex $X^\circ_{p,s}$ corresponding to the subdiagram of the truncated mapping cone diagram consisting of those $A^\circ_t$ and $B^\circ_t$ with $t\equiv s \pmod p$, and all corresponding maps.
\end{proposition}
Note that $\widehat{B}_s\cong \F$ and $B^+_s \cong \F[U, U^{-1}]/U \F[U]$ for all $s$.  Throughout, we use the notation $\tower$ for $\F[U,U^{-1}]/U \F[U]$.  Two examples of the truncated mapping cone diagram can be found in Figure~\ref{fig:truncated}.

\begin{lemma}\label{lem:vs=h-s}
For all $s \in \Z$, the modules $A^\circ_s = H_*(\cA^\circ_s)$ and $A^\circ_{-s} = H_*(\cA^\circ_{-s})$ are isomorphic and under this isomorphism, the maps $v^\circ_s$ and $h^\circ_{-s}$ agree.
\end{lemma}

\begin{proof}
By \cite[Proposition 3.9]{knotinvariants}, we have that the complex obtained from $CFK^\infty(K)$ by reversing the roles of $i$ and $j$ is filtered chain homotopy equivalent to the original complex. Thus, the subquotient complexes $\cA^\circ_s$ and $\cA^\circ_{-s}$ are chain homotopy equivalent, and the following diagram commutes
\[
\begin{CD}
	 \cA^\circ_s @>\mfv^\circ_s>> \cB^\circ_s \\
	@V \cong VV				@V \cong VV\\
	\cA^\circ_{-s} @>\mfh^\circ_{-s}>> \cB^\circ_{-s+p}.
\end{CD}
\]
\end{proof}

Consider the submodule of $A^+_s$ consisting of the image of $U^N$ for arbitrarily large $N$. This submodule is independent of $N$ and isomorphic to  $\tower$. Denote the restrictions of $v^+_s$ and $h^+_s$ to this submodule by $\overline{v}^+_s$ and $\overline{h}^+_s$, respectively.
Define
\begin{align*}
V_s &= \rank (\ker \overline{v}^+_s), \\
H_s &= \rank (\ker \overline{h}^+_s).
\end{align*}
It turns out that for all $s$, both $V_s$ and $H_s$ are finite.  By the lemma above, it follows that $V_s=H_{-s}$.
Under the identifications of $U^N \cdot A^+_s$ and $B^+_s$ with $\tower$, we have that $\overline{v}^+_s$ is given by multiplication by $U^{V_s}$.  A similar result holds for $\overline{h}^+_s$.   

\begin{lemma}[{\cite[Lemma 2.4]{niwu}}]\label{lem:monotonicity}
The $V_s$ form a non-increasing sequence, i.e.,
\[ V_s \geq V_{s+1} \quad \textup{ for all } s \in \Z.\]
Similarly, the $H_s$ form a non-decreasing sequence, i.e.,
\[ H_s \leq H_{s+1} \quad \textup{ for all } s \in \Z.\]
\end{lemma}

\begin{lemma}\label{lem:stominuss}
For all $s \in \Z$, the integers $V_s$ and $V_{-s}$ are related by
\[ V_{-s} = V_s +s. \]
\end{lemma}

\begin{proof}
For $N \geq 2g-1$, the $d$-invariants of $S^3_N(K)$ are determined by the $V_s$ for $|s| \leq \frac{1}{2}(N-1)$ \cite[Theorem 4.4]{knotinvariants}  (cf. \cite[Theorem 5.3]{pap:HLR}). More precisely, consider the quotient map 
\[ C\{\max\{i,j-s\} \geq 0\} \rightarrow  C\{\max\{i,j-s\} \geq 0\} / C\{ i < 0, j \geq s\} = C\{i \geq 0\}, \]
and note that induced map on homology can be identified with $v^+_s: A^+_s \rightarrow B^+_s$. As discussed above, for $k \gg 0$, under the identifications of $U^k \cdot A^+_s$ and $B^+_s$ with $\tower$, $\overline{v}_s^+$ is multiplication by $U^{V_s}$. Since the lowest grading of a non-trivial element in $H_*(C\{i \geq 0\}) \cong HF^+(S^3)$ is zero and the quotient map from $C\{\max\{i,j-s\} \geq 0\}$ to $C\{i \geq 0\}$ preserves grading, the lowest grading of a non-zero element in $U^k \cdot A^+_s$ is $-2V_s$.   We have 
\begin{equation} \label{eqn:dN}
	d(S^3_N(K), [s]) = -2V_s  - s + \frac{4s^2 + N^2 - N}{4N}.
\end{equation}
By \cite[Theorem 1.2]{absgraded}, the $d$-invariants are invariant under Spin$^c$ conjugation, i.e., 
\begin{equation} \label{eqn:dconj}
	d(S^3_N(K), [s]) = d(S^3_N(K), [-s]).
\end{equation}
Then combining \eqref{eqn:dN} and \eqref{eqn:dconj}, we have
\begin{align*}
	0 &= d(S^3_N(K), [s]) - d(S^3_N(K), [-s]) \\
	&= -2V_s  - s + \frac{4s^2 + N^2 - N}{4N} -\Big(-2V_{-s}  - (-s) + \frac{4s^2 + N^2 - N}{4N}\Big) \\
	&= -2(V_s -V_{-s} + s),
\end{align*}
i.e., $V_{-s} = V_s + s$, as desired.
\end{proof}

We conclude this subsection by recalling the definition of $\nu(K)$, first given in \cite[Definition 9.1]{rationalsurgeries}:
$$\nu(K)=\min\{s\mid \v_s \neq 0\}.$$
The knot genus has a description in terms of $\nu(K)$ \cite[Proposition 9.6]{rationalsurgeries}:
\begin{equation}\label{eqgenusnu} g = \max \{\nu(K),\{s\mid \dim \widehat{A}_{s-1} >1\}\}. \end{equation}

\subsection{Some remarks and a proof of Theorem \ref{thm:genusone}}
In order to utilize Heegaard Floer homology to study reducible surgeries, we establish an important structure in the Heegaard Floer homology of a 3-manifold with an $L$-space summand.  

\begin{lemma} 
\label{lem:tperiodic} Let $Y$ be a $3$-manifold such that $Y \cong Y_1\ \#\ Y_2$, where $Y_1$ is an $L$-space and $\left|H^2(Y_2)\right| = r<\infty$.  Then for any $\s\in\text{Spin}^c(Y)$ and $\alpha \in H^2(Y)$, we have $HF^+(Y,\s+r\alpha) \cong HF^+(Y,\s)$ as relatively-graded $\F[U]$-modules.  In particular, $\dim \HFhat(Y,\s + r\alpha) = \dim \HFhat(Y,\s)$.  
\end{lemma}
\begin{proof}  
It is our goal to appeal to the K\"unneth formula for Heegaard Floer homology to study its behavior under connect sums.  However, in order to do this, it is easier to work with the minus flavor of Heegaard Floer homology, since the chain complex $CF^-$ is a finitely-generated complex of free $\F[U]$-modules.  Recall that $HF_-(-Y,\s) \cong HF^+(Y,\s)$ by \cite[Proposition 2.5]{hfpa}, where $HF_-(-Y,\s)$ is the cohomology of the complex $\Hom_{\F[U]}(CF^-(-Y,\s),\F[U])$.  Therefore, we will establish the desired isomorphism by showing that $HF^-(-Y,\s) \cong HF^-(-Y,\s + r\alpha)$, since $\F[U]$ is a PID.  

Let $\s = \s_1\# \s_2$ for $\s_i\in \text{Spin}^c(Y_i)$, $i=1,2$.  
We now apply the K\"unneth formula for $HF^-$ \cite[Theorem 1.5]{hfpa} to see that 
\[
HF^-(-Y,\s) \cong H_*(CF^-(-Y_1,\s_1) \otimes_{\F[U]} CF^-(-Y_2,\s_2))
\]
where the isomorphism is as relatively-graded $\F[U]$-modules.  
Since $Y_1$ is an $L$-space, we also have that $-Y_1$ is an $L$-space (again by \cite[Proposition 2.5]{hfpa}), and thus $HF^-(-Y_1,\s_1) \cong \F[U]$.  Therefore, by the (algebraic) K\"unneth theorem, we have  
\[
HF^-(-Y,\s) \cong HF^-(-Y_2,\s_2).
\]
In particular, $HF^-(-Y,\s)$ is independent of $\s_1$.  Under the identification $H^2(Y)\cong H^2(Y_1)\oplus H^2(Y_2)$, we have $r\alpha = (\gamma,0)$ for some $\gamma \in H^2(Y_1)$.  Therefore, we have $\s + r\alpha = (\s_1 + \gamma) \# \s_2$, which implies   
\[
HF^-(-Y,\s + r\alpha) \cong HF^-(-Y_2,\s_2).
\]
Therefore, $HF^+(Y,\s + r\alpha) \cong HF^+(Y,\s)$.  Since $\HFhat$ is determined by $HF^+$ when working over $\F$, the desired equality for $\dim \HFhat$ holds as well.
\end{proof}



To complete this section, we give a proof of Theorem \ref{thm:genusone}.
\begin{proof}[Proof of Theorem \ref{thm:genusone}]  Let $K$ be a genus one knot with reducing slope $p$.  Up to mirroring, we may assume $p>0$.  Since $\hat{v}_s$ is an isomorphism for $s \geq 1$ and $\hat{h}_s$ is an isomorphism for $s \leq -1$, Proposition \ref{prop:truncatedmcf} yields 
\[\HFhat(S^3_p(K),[s]) \cong \left\{\begin{array}{lcl} 
\hat{A}_0 & \hspace{2pt} & \text{if } s\equiv 0 \pmod p, \\
\F & \hspace{2pt} & \text{if } s\not\equiv 0 \pmod p.
\end{array}\right.\]
Since $S^3_p(K)$ is reducible, there is a non-trivial lens space connected summand.  Since lens spaces are $L$-spaces, Lemma~\ref{lem:tperiodic} implies there must be a Spin$^c$ structure $[s]\neq[0]$ for which $\HFhat(S^3_p(K),[s])\cong \Ahat_0$.  This implies that $\Ahat_0\cong \F$, so the knot admits an $L$-space surgery, and hence is fibered by \cite{pap:ghigginifibered}.  Recall that the only genus one fibered knots are the figure eight knot and the trefoil.  One may check that for the figure eight knot, $\dim \Ahat_0 = 3$, which is a contradiction.  On the other hand, the trefoil is a torus knot, and hence satisfies the Cabling Conjecture by \cite{pap:moser}.
\end{proof}

%% file: lspace.tex
\section{$L$-space knots}\label{sec:lspace}
In this section we study the mapping cone formula for $L$-space knots and prove Theorem \ref{thm:maintheorem}.  Let $K$ be a genus $g$ knot in $S^3$ with an $L$-space surgery.  Up to mirroring, we may assume that this surgery coefficient is positive.  However, the purported reducing slope for $K$ may be negative.  In light of Theorem \ref{thm:genusone}, assume $g \geq 2$.

In the present situation, we obtain additional information about the truncated mapping cone diagram.  First, it follows from \cite[Theorem 4.4]{knotinvariants} and a surgery exact triangle argument (using \cite[Theorem 9.16]{hfpa}) that 
\begin{equation} 
	\widehat{A}_s \cong \F \qquad \textup{and} \qquad A^+_s \cong \tower
\end{equation}
for each $s \in \Z$. This implies, by \eqref{eqgenusnu}, that $\nu(K) = g$.  By definition, all the maps $\v_s$ in the truncated mapping cone diagram vanish; one can also deduce that the $\h_s$ vanish as well.  As for the maps $v^+_s$, recall that for $s \geq g$ the map 
\[v^+_s: A^+_s \rightarrow B^+_s \]
is an isomorphism. For $s<g$, since $V_s$ is finite and the map $v^+_s$ is $U$-equivariant, we must have that $v^+_s$ is a surjection.

\subsection{$\HFhat$ and $p \mid 2g-1$}
\noindent The goal of this subsection is to prove the following preliminary result using $\HFhat$.  
\begin{proposition}\label{prop:mainprop}
Let $K$ be a hyperbolic $L$-space knot.  If $p$ is a reducing slope, then $p \mid (2g-1)$.  
\end{proposition}
With the mapping cone formula one can easily compute the total dimension of $\HFhat(S^3_p(K))$ (see \cite[Proposition 9.5]{rationalsurgeries}).  However, in order to prove Proposition~\ref{prop:mainprop}, we need to keep track of the Spin$^c$-grading.  Recall that the identification $\text{Spin}^c(S^3_p(K))\cong \Z/p\Z$ is fixed.
	
\begin{lemma}
\label{lem:lsknotsurg} Let $K$ be a knot with a positive $L$-space surgery, and let $p$ be any integral slope on $K$ satisfying \eqref{genusbound}.  Let $k\equiv 2g-1 \pmod p$ so that $0\leq k <|p|$.  If $p>0$, for $s\in\Z$ such that $g-k\leq s< g-k+p$,
\begin{equation*}
\dim\HFhat(S^3_p(K),[s]) = \left\{\begin{array}{ll}
2\left\lfloor\frac{2g-1}{p}\right\rfloor +1& \textup{if } g-k\leq s < g , \\ \\
2\left\lfloor\frac{2g-1}{p}\right\rfloor -1& \textup{if } g\leq s < g-k+p.
\end{array}\right.\end{equation*}
If $p<0$, for $s\in\Z$ such that $g-k\leq s< g-k+|p|$,
\begin{equation*}
\dim\HFhat(S^3_p(K),[s]) = \left\{\begin{array}{ll}
2\left\lfloor\frac{2g-1}{|p|}\right\rfloor + 3&\textup{if } g-k\leq s < g , \\ \\
2\left\lfloor\frac{2g-1}{|p|}\right\rfloor + 1& \textup{if } g\leq s < g-k+|p|. 
\end{array}\right.\end{equation*}
\end{lemma}
We are interested in Lemma \ref{lem:lsknotsurg} since it immediately establishes the following.

\begin{proposition} \label{prop:countingHF}
Let $K$ be a knot with a positive $L$-space surgery, and let $p$ be any integral slope on $K$ satisfying \eqref{genusbound}.   If $p$ does not divide $2g-1$, the mod $p$ residue class of $g$ is the unique class such that $\dim\HFhat(S^3_p(K),[g-1])>\dim\HFhat(S^3_p(K),[g])$.
\end{proposition}

\begin{figure} [t]
\begin{center}
\begin{subfigure}{.99\textwidth}
\labellist
\pinlabel $\mathcal{S}_{-1}$ at 170 87
\pinlabel $\mathcal{S}_0$ at 333 87
\pinlabel $\mathcal{S}_1$ at 458 87
\endlabellist
\includegraphics[width=.99\textwidth]{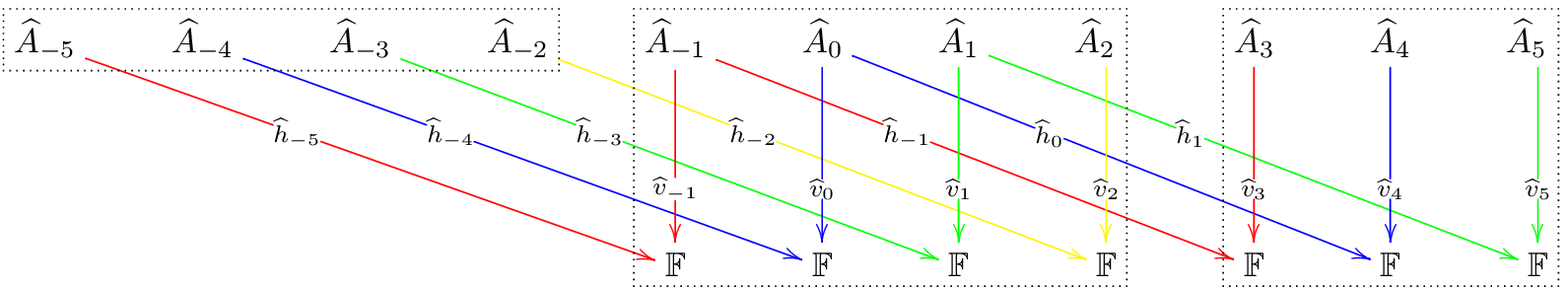}
\caption{The truncated mapping cone diagram for $4$-surgery on a genus 6 $L$-space knot.  In this example, $t_0 = -1$ and $m = 1$. }
\end{subfigure}
\begin{subfigure}{.99\textwidth}
\label{fig:-2surgery}
\labellist
\small\hair 2pt
\pinlabel $\mathcal{S}_{-1}$ at 5 18
\pinlabel $\mathcal{S}_0$ at 140 86
\pinlabel $\mathcal{S}_1$ at 238 86
\pinlabel $\mathcal{S}_2$ at 324 86
\pinlabel $\mathcal{S}_3$ at 370 86
\endlabellist
\includegraphics[width=.99\textwidth]{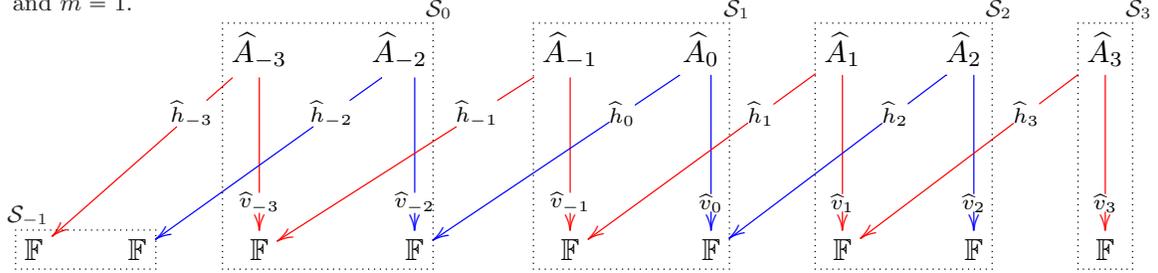}
\caption{The truncated mapping cone diagram for $-2$-surgery on a genus 4 $L$-space knot.  In this example, $t_0 = -3$ and $m = 3$.}
\end{subfigure}
\caption{Two examples of truncated mapping cone diagrams.  Each color in these diagrams corresponds to a Spin$^c$ structure on $S^3_p(K)$.}\label{fig:truncated}
\end{center}
\end{figure}

\begin{proof}[Proof of Lemma \ref{lem:lsknotsurg}]
Let $t_0 = 1-g + \max\{0,p\}$.  Then $t_0$ is the smallest value of $t$ such that both the domain and codomain of the map $v_t$ appear in the truncated mapping cone diagram.  Consider the following partition of the indices that occur in the diagram: 
$$\mathcal{I}_i=\left\{t\mid \left\lfloor\frac{t-t_0}{|p|}\right\rfloor=i\right\}.$$
If $\mathcal{I}_i$ is nonempty, then $-1\leq i\leq m$ where
$$m =\left\{ \begin{array}{lcr} \lfloor\frac{2g-1}{p}\rfloor -1  & \text{ if } & p>0,\\
\lfloor\frac{2g-1}{|p|}\rfloor & \text{ if } & p<0.\end{array}\right.$$
Fix an integer $s$.  For $-1\leq i\leq m$, there is at most one $t\in \mathcal{I}_i$ with $t\equiv s\pmod p$.

Let $\mathcal{D}$ denote the object set of the truncated mapping cone diagram.  Since the maps $\hat{v}_t$ and $\hat{h}_t$ in the diagram vanish and $\dim \Ahat_t=\dim \Bhat_t = 1$ for all $t$, Proposition \ref{prop:truncatedmcf} implies that computing the dimension of $\HFhat(S^3_p(K),[s])$ reduces to counting the number of objects $\Ahat_t$ and $\Bhat_t$ in $\mathcal{D}$ with $t\equiv s \pmod p$.  To facilitate this, we partition $\mathcal{D}$ into sets $\mathcal{S}_i=\left\{\Ahat_t,\Bhat_t\in \mathcal{D}\mid t\in\mathcal{I}_i\right\}$.  See Figure~\ref{fig:truncated} for two explicit examples.  
For all $0 \leq i \leq m-1$, the contribution of $\mathcal{S}_i$ to $\dim\widehat{HF}(S^3_p(K),[s])$ is two.  The contribution of $\mathcal{S}_{-1}$ to $\dim\HFhat(S^3_p(K),[s])$ is one.  Finally, the contribution of $\mathcal{S}_{m}$ is two if there exists $t \equiv s \pmod p$ such that $g - k \leq t < g$ and zero otherwise.  The result now follows from summing the contributions.
\end{proof}

\begin{proof}[Proof of Proposition \ref{prop:mainprop}]
Let $K$ be a hyperbolic knot with a positive $L$-space surgery such that $S^3_p(K)\cong L(m,n)\ \#\ R$ for $L(m,n),R\not\cong S^3$.  Then $p$ satisfies \eqref{genusbound}.  Let $|H^2(R)|=r$.  Clearly $r < p$.  Suppose $p$ does not divide $2g-1$.  By Proposition \ref{prop:countingHF}, the mod $p$ residue class of $g$ is the unique class such that 
$$\dim\HFhat(S^3_p(K),[g-1])>\dim\HFhat(S^3_p(K),[g]).$$  By Lemma \ref{lem:tperiodic}, we must also have
$$\dim\HFhat(S^3_p(K),[g+r-1])>\dim\HFhat(S^3_p(K),[g+r]).$$
Uniqueness implies that $r$ is a multiple of $p$.  This is a contradiction.\end{proof}

\subsection{$HF^+$ and the proof of Theorem \ref{thm:maintheorem}}
In the following lemmas, we give a description of $\HFred(S^3_p(K))$ in terms of the maps $v^+_t$ and $h^+_t$.  As discussed, since $K$ is an $L$-space knot, each object in the truncated mapping cone diagram (either $A^+_t$ or $B^+_t$) is isomorphic to $\cT^+$.  Recall that under these identifications we have that $v^+_t(x) = U^{V_t}x$ and $h^+_t(x) = U^{H_t}x$.  In this setting, $\ker(v^+_t + h^+_t)$ is isomorphic to $\F[U]/U^{\min\{V_t,H_t\}}$.  Finally, we recall that for a rational homology sphere $Y$, $HF^\infty(Y,\spinc) \cong \F[U,U^{-1}]$ for all $\spinc$.  In particular, this implies that $U^k \cdot HF^+(Y,\spinc)$ is isomorphic to $\tower$ for $k$ sufficiently large, and that under this identification, $\HFred(Y,\spinc) \cong HF^+(Y,\spinc)/\tower$.  Note that Lemma \ref{lem:HFredpos} below is essentially a special case of \cite[Proposition 3.6]{nizhang}; we include it to obtain explicit contact with the relative gradings on $\HFred$.

\begin{lemma}\label{lem:HFredpos}
Let $K$ be an $L$-space knot. If $p>0$, $p\mid(2g-1)$, and $p \neq 2g-1$, then for $|s| < \frac{p}{2}$, we have
\[ \HFred(S^3_p(K), [s]) \cong \ker h^+_k \oplus \ker v^+_\ell \oplus \bigoplus_{\substack{\frac{p}{2} < |t| \leq g-1-p \\ t \equiv s \pmod p}} \ker (v^+_t +h^+_t),    \]
where $k \equiv s \pmod p$ with $1-g \leq k < 1-g+p$ and $\ell \equiv s \pmod p$ with $g-1-p < \ell \leq g-1$.\end{lemma}

\begin{remark}
A similar result holds for any $p>0$ such that $p \leq 2g-1$, but for notational convenience, we restrict to the case $p \mid 2g-1$. In particular, when $p$ is close to $2g-1$ (i.e., $2p > 2g-1$), for certain values of $s$, the complex $X^+_{p, s}$ may consist solely of $A^+_s$.
\end{remark}

\begin{proof}
Fix $|s| < \frac{p}{2}$.  We consider the truncated mapping cone diagram $X^+_{p,s}$ for the Spin$^c$ structure $[s]$. Recall that $X^+_{p,s}$ consists of $A^+_t$ for $t \equiv s \pmod p$ and $1-g \leq t \leq g-1$, and $B^+_t$ for $t \equiv s \pmod p$ and $1-g+p \leq t \leq g-1$ together with the maps between them.

Consider the submodule of $H_*(X^+_{p,s})$ consisting of the image of arbitrarily large powers of $U$. We claim that any element in this submodule has non-zero projection to $A^+_s$. Indeed, identify $A^+_t$ with $\F[U, U^{-1}]\langle \omega_t \rangle / U\F[U]\langle \omega_t \rangle$. Fix $N \gg 0$, and consider the chain
\[ \eta_s =  \left( \sum_{\substack{1-g \leq t < -\frac{p}{2} \\ t \equiv s \pmod p}}U^{-H_t + V_s - m_t - N}   \omega_t \right) + U^{-N}\omega_s +   \left( \sum_{\substack{\frac{p}{2} < t \leq g-1 \\ t \equiv s \pmod p}}U^{H_s - V_t + n_t - N} \omega_t \right), \]
where 
\[ m_t = \sum_{\substack{t< i< -\frac{p}{2} \\ i \equiv s \pmod p}} i \qquad \textup{and} \qquad n_t = \sum_{\substack{\frac{p}{2} < i < t \\ i \equiv s \pmod p}} i. \]
Note that for $1-g \leq t < -\frac{p}{2}$, the sum $-H_t+V_s-m_t$ is non-negative since $V_s -H_t \geq 0$ by Lemmas \ref{lem:vs=h-s}, \ref{lem:monotonicity}, and \ref{lem:stominuss}. Similarly, for $\frac{p}{2} < t \leq g-1$, we have that $H_s - V_t + n_t$ is non-negative. 

We consider the action of $U$ on $\eta_s$. It is straightforward to verify that $U^n \cdot \eta_s$ is a cycle and non-zero in $H_*(X^+_{p,s})$ for any $0\leq n \leq N$, and that $U^n \cdot \eta_s=0$ for $n>N$. (Recall that $H_t -V_t =t$ for all $t \in \Z$ by Lemmas \ref{lem:vs=h-s} and \ref{lem:stominuss}.) In particular, for all $0\leq n \leq N$, the projection of the cycle $U^{n} \cdot \eta_s$ to $A^+_s$ is $U^{n-N} \cdot \omega_s$. By letting $N$ grow arbitrarily large, it follows that any element of $H_*(X^+_{p,s})$ which is in the image of arbitrarily large powers of $U$ has non-zero projection to $A^+_s$. This completes the proof of the claim.

Now consider the subcomplex $R_{p,s} \subset X^+_{p, s}$ consisting of all of the $A^+_t$'s, except for $A^+_s$, and all of the $B^+_t$'s. This subcomplex determines a short exact sequence
\[  0 \rightarrow R_{p,s} \rightarrow X^+_{p,s} \rightarrow A^+_s \rightarrow 0. \]
By the preceding paragraph, the induced exact triangle on homology induces an injection from $U^k \cdot H_*(X^+_{p,s})$  to $A^+_s$ for arbitrarily large $k$.  As an injective $\F[U]$-module map from $\tower$ to $\tower$ is an isomorphism, we must have that the induced map from $U^k \cdot H_*(X^+_{p,s})$ to $A^+_s$ is an isomorphism.  Therefore, $\HFred(S^3_p(K), [s]) \cong H_*(R_{p,s})$.

We now show that the homology $H_*(R_{p,s})$ is as described in the statement of the lemma. We claim the kernel of the differential on $R_{p, s}$ is 
\[ \bigoplus_{\substack{1-g+p \leq t \leq g-1 \\ t \equiv s \pmod p}} B^+_t \oplus \ker h^+_k \oplus \ker v^+_\ell \oplus \bigoplus_{\substack{1-g+p \leq t < \frac{p}{2} \\ t \equiv s \pmod p}} \ker (v^+_t +h^+_t)  \oplus \bigoplus_{\substack{\frac{p}{2} < t \leq g-1-p \\ t \equiv s \pmod p}} \ker (v^+_t +h^+_t). \]
Indeed, any element in the kernel of the differential contained in $\oplus_{1-g \leq t < \frac{p}{2}} A^+_t$ must be a linear combination of elements in $\ker h^+_k$ and $\ker(v^+_t + h^+_t)$ for $1-g+p \leq t  < \frac{p}{2}$ since the sequences $\{V_s\}$ and $\{H_s\}$ are non-increasing and non-decreasing, respectively, by Lemma \ref{lem:monotonicity}. A similar statement holds for any element in the kernel of the differential contained in $\oplus_{\frac{p}{2} <  t \leq g-1} A^+_t$. Each $B^+_t$ lies in the image of the differential on $R_{p,s}$. This completes the proof of the lemma.
\end{proof}

We now consider negative surgeries.

\begin{lemma} \label{lem:HFredneg}
Let $K$ be an $L$-space knot. If $p<0$, $p\mid(2g-1)$, and $p \neq 1-2g$, then for $s \in \Z$, we have
\[ \HFred(S^3_p(K), [s]) \cong \bigoplus_{\substack{1-g \leq t \leq g-1 \\ t \equiv s\pmod p}} \ker (v^+_t +h^+_t). \]
\end{lemma}

\begin{proof}
As in the preceding proof, we consider the truncated mapping cone diagram $X^+_{p,s}$ for the Spin$^c$ structure $[s]$. The diagram $X^+_{p,s}$ consists of $A^+_t$ for $t \equiv s \pmod p$ and $1-g \leq t \leq g-1$, and $B^+_t$ for $t \equiv s \pmod p$ and $1-g+p \leq s \leq g-1$, together with the maps between them. Each $A^+_t$ and $B^+_t$ is isomorphic to $\cT^+$ since $K$ is an $L$-space knot.

The kernel of the differential on $X^+_{p,s}$ is
\[ \bigoplus_{\substack{1-g+p \leq t \leq g-1 \\ t \equiv s \pmod p}} B^+_t \oplus \bigoplus_{\substack{1-g \leq t \leq g-1 \\ t \equiv s \pmod p}} \ker(v^+_t + h^+_t) .\]
Indeed, any element in the kernel of the differential contained in $\oplus_{1-g \leq t \leq g-1} A^+_t$ must be a linear combination of elements in $\ker(v^+_t + h^+_t)$ for $1-g \leq t \leq g-1$ since the sequences $\{V_s\}$ and $\{H_s\}$ are non-increasing and non-decreasing, respectively, by Lemma \ref{lem:monotonicity}.  Since the image of the differential is contained in 
\[
\bigoplus_{\substack{1-g+p \leq t \leq g-1 \\ t \equiv s \pmod p}} B^+_t,
\]
we can identify 
\[
\bigoplus_{\substack{1-g \leq t \leq g-1 \\ t \equiv s\pmod p}} \ker (v^+_t +h^+_t)
\] 
with a submodule of $\HFred(S^3_p(K),[s])$. 

Let $t_0$ be such that $1-g +p \leq t_0 < 1-g$ and $t_0 \equiv s \pmod p$. Let $n=V_{t_0+|p|}-H_{t_0+|p|}$. One can check that the image of the differential contained in $B^+_{t_0}$ is $\ker U^n$ and so $B^+_{t_0}/\ker U^n$ is not in the image of the differential. Furthermore, any element of $\bigoplus_{\substack{1-g \leq t \leq g-1 \\ t \equiv s \pmod p}} B^+_t$ that is not in the image of the differential is homologous to a (possibly trivial) element of $B^+_{t_0}/\ker U^n$. In particular, we can identify $B^+_{t_0}/\ker U^n$ with the image of $U^N$ on $HF^+(S^3_p(K), [s])$ for large $N$. 

Thus, we conclude that
\[ \HFred(S^3_p(K), [s]) \cong \bigoplus_{\substack{1-g \leq t \leq g-1 \\ t \equiv s \pmod p}} \ker (v^+_t +h^+_t), \]
as desired.
\end{proof}

We also compute $HF^+(S^3_p(K))$ in terms of the $v^+_t$ and $h^+_t$ when $p=1-2g$.

\begin{lemma}
Let $K$ be an $L$-space knot. If $p=1-2g$, then for $|s| < g$, we have isomorphisms 
\[ 
HF^+(S^3_p(K), [s]) \cong \left\{ \begin{array}{lcr} B^+_s  \oplus \ker(v^+_s + h^+_s)& \text{ if } & s< 0,\\
B^+_{s-|p|} \oplus \ker(v^+_s+h^+_s) & \text{ if } & s \geq 0. \end{array} \right. 
\]
\end{lemma}

\begin{proof}
Fix $|s| < g$.  We consider the truncated mapping cone diagram $X^+_{p,s}$, which consists of $A^+_s$, $B^+_s$, and $B^+_{s-|p|}$, together with the map $h^+_s$ from $A^+_s$ to $B^+_{s-|p|}$ and the map $v^+_s$ from $A^+_s$ to $B^+_s$.

The kernel of the differential is
\[ B^+_{s-|p|} \oplus B^+_s \oplus \ker(v^+_s+h^+_s). \]
If $1-g \leq s <0$, then by Lemmas \ref{lem:vs=h-s} and \ref{lem:stominuss}, we have that $H_s<V_s$. Thus, no element in $B^+_s$ is  in the image of the differential, and any element in $B^+_{s-|p|}$ is homologous to a (possibly trivial) element in $B^+_s$. In particular, we can identify $B^+_s$ with $U^N \cdot H_*(X_{p,s})$ for $N \gg 0$.  Similarly, if $0 \leq s \leq g-1$, then $H_s \geq  V_s$, and we can identify $B^+_{s-|p|}$ with $U^N \cdot H_*(X_{p,s})$ for $N \gg 0$. Furthermore, any element in $B^+_s$ is either homologous to something in $B^+_{s-|p|}$ or is in the image of the differential.  This completes the proof.
\end{proof}

In order to prove Theorem~\ref{thm:maintheorem}, we must analyze $HF^+(S^3_p(K),[s])$ as a relatively-graded $\F[U]$-module.  As discussed, for an $L$-space knot, each $A^+_s$ is isomorphic to $\tower$ and under this identification, we have $\ker(v^+_s + h^+_s) \cong \F[U]/U^{\min\{V_s,H_s\}}$.  For each $t$ with $|t| \leq g-1$, we consider the elements $x_t$, $y_t$, and $z_t$ in $A^+_t$ which correspond to $U^{-\max\{V_t,H_t\}}$, $U^{-\min\{V_t,H_t\}}$, and $U^{-\min\{V_t,H_t\}+1}$ respectively under the isomorphism $A^+_t \cong \tower$.  Note that $x_t,y_t \not \in \ker(v^+_t + h^+_t)$, while $z_t$ is the unique non-zero element of $\ker(v^+_t + h^+_t)$ with largest relative grading.  Since $K$ is a non-trivial $L$-space knot, we have $\min\{V_t,H_t\} > 0$ for all $|t| \leq g-1$ (combine \cite[Proof of Lemma 8.1]{rationalsurgeries} and \cite[Theorem 1.2]{genusdetection}).  Therefore, $z_t$ always exists for $|t| \leq g-1$.  

We define an auxiliary object  
\begin{align*}
\HFdiamond(S^3_p(K),[s]) &= \coker \big(U:HF^+(S^3_p(K),[s]) \to HF^+(S^3_p(K),[s])\big)\\
&\cong \coker\big(U:\HFred(S^3_p(K),[s]) \to \HFred(S^3_p(K),[s])\big).  
\end{align*}

Since $\HFred(S^3_p(K),[s])$ is finite-dimensional over $\F$, we have that $\HFdiamond(S^3_p(K),[s])$ is as well.  The absolute grading on $HF^+(S^3_p(K),[s])$ induces absolute gradings on both $\HFdiamond(S^3_p(K),[s])$ and on the truncated mapping cone.  We will use $\gr$ to denote the corresponding grading on the truncated mapping cone.  By Lemmas~\ref{lem:HFredpos} and \ref{lem:HFredneg}, we have that when $p > 0$, 
\begin{equation}\label{eqn:positivediamond}
\dim \HFdiamond_n(S^3_p(K),[s]) = \# \{ t \mid  t \equiv s \ (\text{mod }p), \ \frac{p}{2} < |t| \leq g-1, \ \gr(z_t) = n\},
\end{equation}
and when $p < 0$, 
\begin{equation}\label{eqn:negativediamond}
\dim \HFdiamond_n(S^3_p(K),[s]) = \# \{ t \mid t \equiv s \, (\text{mod }p), \ \gr(z_t) = n\}.
\end{equation}
Therefore, studying $\HFdiamond$ as a graded vector space corresponds to understanding the gradings of the $z_t$.

\begin{lemma}\label{lem:grmaxtomin}
For each $t \in \mathbb{Z}$, we have $\gr(x_t) - \gr(y_t) = 2|t|$.  
\end{lemma}
\begin{proof}
Combining the definition of $x_t$ and $y_t$ with Lemmas~\ref{lem:vs=h-s}, \ref{lem:monotonicity}, and \ref{lem:stominuss}, we see 
\begin{align*}
\gr(x_t) - \gr(y_t) &= 2(\max\{V_t,H_t\} - \min\{V_t,H_t\}) \\
&= 2|V_t - H_t| \\
&= 2|V_t - V_{-t}| \\
&= 2|t|. 
\end{align*}
\end{proof}

\begin{lemma}\label{lem:grhtov}
Suppose that $1-g \leq t$, $t+p \leq g-1$.  If $p > 0$, then 
\begin{equation}
\gr(z_{t+p}) - \gr(z_t) = \left \{ \begin{array}{ll} 2t & \textup{ if }  t \geq 0 \\ -2(t+p) & \textup{ if } t+p \leq 0 \\ 0 & \textup{ if } t \leq 0 \leq t+p. \end{array} \right.    
\end{equation}
If $p < 0$, then
\begin{equation}
\gr(z_{t+p}) - \gr(z_t) = \left \{ \begin{array}{ll} 2t & \textup{ if }  t+p \geq 0 \\ 2(t+p) & \textup{ if } t \leq 0\\ 2(2t+p) & \textup{ if } t+p \leq 0 \leq t. \end{array} \right.    
\end{equation}
\end{lemma}  
\begin{proof}
Since $\gr(y_t) - \gr(z_t) = 2$ for all $t$, it suffices to establish the desired equalities for $\gr(y_{t+p}) - \gr(y_t)$ instead.  

We prove the lemma in the case that $p>0$, as the argument for $p < 0$ is similar.  First, let $t\leq 0 \leq t+p$.  We have $V_t \geq H_t$ and $H_{t+p} \geq V_{t+p}$ by Lemma~\ref{lem:monotonicity}.  By the definition of $y_t$ and $y_{t+p}$, this implies that $h^+_t(y_t)$ and $v^+_{t+p}(y_{t+p})$ both correspond to 1 under the identification $B^+_{t+p} \cong \tower$, and thus $h^+_t(y_t) = v^+_{t+p}(y_{t+p})$.  Since the differential in the truncated mapping cone lowers the relative grading by one and both $y_t$ and $y_{t+p}$ are grading homogeneous elements, we have that $\gr(y_t) = \gr(y_{t+p})$.  

Now, suppose $t \geq 0$.  In this case, we have $H_t \geq V_t$ and $H_{t+p} \geq V_{t+p}$.  By an argument similar to the previous case, we have $h^+_t(x_t) = v^+_{t+p}(y_{t+p})$.  This implies $\gr(x_t) = \gr(y_{t+p})$.  The result now follows from Lemma~\ref{lem:grmaxtomin}.  

Finally, let $t \leq t+p \leq 0$.  Now, $V_t \geq H_t$ and $V_{t+p} \geq H_{t+p}$.  We see in this case that $h^+_t(y_t) = v^+_{t+p}(x_{t+p})$.  Therefore, $\gr(y_t) = \gr(x_{t+p})$.  The result again follows from Lemma~\ref{lem:grmaxtomin}.         
\end{proof}

Now suppose that $S^3_p(K) \cong L(m, q) \# R$, where $|H^1(R)| = r$.  By Lemma~\ref{lem:tperiodic}, we know that $\HFdiamond(S^3_p(K),[s])$ and $\HFdiamond(S^3_p(K),[s+r])$ are isomorphic as relatively-graded $\F$-vector spaces.  Therefore, we would like to compare this fact to the gradings computed in Lemma~\ref{lem:grhtov}.  Define 
\[
\grbottom_{[s]} = \min\{n \mid \HFdiamond_n(S^3_p(K),[s]) \neq 0\}
\]  
and 
\[
\grtop_{[s]} = \max\{n \mid \HFdiamond_n(S^3_p(K),[s]) \neq 0\}.  
\]
Therefore, for all $n,s \in \Z$, we have 
\begin{equation}\label{eqn:tperiodicbottom}
\dim \HFdiamond_{\grbottom_{[s]} + n}(S^3_p(K),[s]) = \dim \HFdiamond_{\grbottom_{[s+r]} + n}(S^3_p(K),[s+r])  
\end{equation}
and 
\begin{equation}\label{eqn:tperiodictop}
\dim \HFdiamond_{\grtop_{[s]} + n}(S^3_p(K),[s]) = \dim \HFdiamond_{\grtop_{[s+r]} + n}(S^3_p(K),[s+r]).  
\end{equation}
With this, we are ready for the penultimate step towards the proof of Theorem~\ref{thm:maintheorem}.  

\begin{proposition}\label{prop:penultimate}
Let $K$ be a hyperbolic $L$-space knot.  If $S^3_p(K)$ is a reducible manifold, then $p = \pm (2g-1)$.   
\end{proposition}
\begin{proof}
We suppose that $K$ is a hyperbolic $L$-space knot with the property that $S^3_p(K) = L(m,n) \# R$, where $|H^2(R)| = r$.  We argue by contradiction, supposing that $|p| < 2g-1$.  By Proposition~\ref{prop:mainprop}, we have $p \mid 2g-1$.  Combined with \eqref{genusbound}, this implies $1<|p| \leq \frac{2g-1}{3}$ and  $m \geq 3$.  In particular, $2g-1$ is not prime, so $g \geq 5$.  These conditions guarantee that for each Spin$^c$ structure there are exactly $\frac{2g-1}{|p|} \geq 3$ objects of the form $A^+_t$ in the truncated mapping cone.  

We first consider the case that $p > 0$.  By Lemma~\ref{lem:grhtov}, we have that $\gr(z_{-p}) = \gr(z_0) = \gr(z_p)$ (since $\frac{2g-1}{p} \geq 3$, these elements all exist in the truncated mapping cone).  Furthermore, for any other $s \equiv 0 \pmod{p}$, Lemma~\ref{lem:grhtov} guarantees $\gr(z_s) > \gr(z_0)$.  By definition, we have that $\grbottom_{[0]} = \gr(z_p)$.  Because $z_0$ does not contribute to $\HFred$ by Lemma~\ref{lem:HFredpos}, we apply \eqref{eqn:positivediamond} to conclude 
\[
\dim \HFdiamond_{\grbottom_{[0]}}(S^3_p(K),[0]) = 2.
\]
We now do the analogous computation for the Spin$^c$ structure $[-r]$.  Because $-r < 0  < p-r$, Lemma~\ref{lem:grhtov} implies that $\gr(z_{-r}) = \gr(z_{p-r})$ and $\gr(z_t) > \gr(z_{-r})$ for all other $t \equiv -r \pmod{p}$.  Therefore, \eqref{eqn:positivediamond} implies 
\[
\dim \HFdiamond_{\grbottom_{[-r]}}(S^3_p(K),[-r]) = 1, 
\]
since $z_{-r}$ does not contribute to $\HFred$.  
This contradicts \eqref{eqn:tperiodicbottom}.

We will use a similar argument for the case that $p < 0$.  By Lemma~\ref{lem:grhtov}, we have that 
\[
\gr(z_0) - \gr(z_p) = \gr(z_0) - \gr(z_{-p}) = 2|p|.
\]  
Again, we remind the reader that the conditions on $p$ guarantee that all three of these elements appear in the truncated mapping cone.  Furthermore, for any other $t \equiv 0 \pmod{p}$, we have that $\gr(z_0) - \gr(z_t) > 2|p|$.  Thus, $\grtop_{[0]} = \gr(z_0)$, since for $p<0$, the element $z_0$ contributes to $\HFred$ by Lemma~\ref{lem:HFredneg}.  By \eqref{eqn:negativediamond}, we have that 
\[
\dim \HFdiamond_{\grtop_{[0]} - 2|p|}(S^3_p(K),[0]) = 2. 
\] 
We now will compare this to $\HFdiamond_{\grtop_{[r]}-2|p|}(S^3_p(K),[r])$.  We claim that 
\[
\dim \HFdiamond_{\grtop_{[r]} - 2|p|}(S^3_p(K),[r]) \neq 2.
\]
This will follow by again computing the relative gradings of $z_t$ for $t \equiv r \pmod{p}$.  In particular, Lemma~\ref{lem:grhtov} implies $\gr(z_r) - \gr(z_{r+p}) = -2(2r+p)$, while $\gr(z_r) - \gr(z_{r-p}) = 2(r-p)$.  Furthermore, for any other $t \equiv r \pmod{p}$, Lemma~\ref{lem:grhtov} implies $\gr(z_r) - \gr(z_{t}) > 2|p|$.  In particular, we have that $\gr(z_r) = \grtop_{[r]}$.  Therefore, if $\dim \HFdiamond_{\grtop_{[r]} - 2|p|}(S^3_p(K),[r])$ were to equal 2, we would need $-2(2r+p) = 2(r-p) = 2|p|$.  This would imply $r = 0$, contradicting the fact that $r \mid p$.  
We now conclude that 
\[
\dim \HFdiamond_{\grtop_{[0]} - 2|p|}(S^3_p(K),[0]) \neq \dim \HFdiamond_{\grtop_{[r]} - 2|p|}(S^3_p(K),[r]), 
\]        
contradicting \eqref{eqn:tperiodictop}.
\end{proof}

\begin{proof}[Proof of Theorem~\ref{thm:maintheorem}]
Let $K$ be a hyperbolic $L$-space knot with reducing slope $p$.  Suppose for contradiction that $p \neq 2g-1$.  In light of Proposition~\ref{prop:penultimate}, it suffices to assume that the reducing slope is $p = {1-2g}$.  Suppose that $S^3_{1-2g}(K) = L(m,n) \# R$, where $|H^2(R)| = r$.  We assume the same notation as used throughout this section.  We fix $0 \leq s \leq g-1$.  Lemma~\ref{lem:HFredneg} gives that for each $s \in \Z$ with $0 \leq s \leq g-1$, we have 
\[
HF^+(S^3_{1-2g}(K),[s]) \cong B^+_{s+1-2g} \oplus \ker(v^+_s + h^+_s) \cong \tower \oplus \F[U]/U^{\min\{V_s,H_s\}}.  
\]
Therefore, there is only one $z_t$ to consider for each $0 \leq s \leq g-1$, namely $z_s$.  Now, for any $t$, let $b_t$ denote the element of $B^+_t$ corresponding to 1 under the identification $B^+_t \cong \tower$.  We have $\gr(b_{s+1-2g}) = d(S^3_{1-2g}(K),[s])$; call this grading level $d_s$.  We would like to determine the difference in the gradings of $z_s$ and $b_{s+1-2g}$, by studying the relative gradings in the truncated mapping cone.  Recall that the components of the differential on the truncated mapping cone are given by $v^+_s + h^+_s$, and thus these lower relative grading by one.  Since $s \geq 0$, we have $H_s \geq V_s$, which implies $h^+_s(x_s) = b_{s+1-2g}$, and thus $\gr(x_s) - \gr(b_{s+1-2g}) = 1$.  On the other hand, by Lemma~\ref{lem:grmaxtomin}, 
\[
\gr(x_s) - \gr(z_s) = \gr(x_s) - \gr(y_s) + 2 = 2s + 2.
\]
Therefore, $\gr(b_{s+1-2g}) - \gr(z_s) = 2s + 1$, and thus, $\gr(z_s) = d_s - 2s - 1$.  We have, for $0 \leq s \leq g-1$, that as absolutely graded $\F[U]$-modules,
\[
HF^+(S^3_{1-2g}(K),[s]) \cong \tower_{(d_s)} \oplus (\F[U]/U^{\min\{V_s,H_s\}})_{\{d_s - 2s - 1\}},
\]
where we use the notation $(\F[U]/U^k)_{\{q\}}$ to mean that the degree of the {\em highest}-grading non-zero element (namely 1) is $q$.  It is now easy to see that $HF^+(S^3_{1-2g}(K),[0])$ and $HF^+(S^3_{1-2g}(K),[r])$ are not isomorphic as relatively-graded $\F[U]$-modules, contradicting Lemma~\ref{lem:tperiodic}.           
\end{proof}

%% file: doublereducibles.tex
\section{Knots with two reducible surgeries}\label{sec:tworeducibles}

\begin{proof}[Proof of Theorem \ref{thm:tworeducibles}]  Suppose that $K$ is a knot of genus at most two with two reducing slopes.  As discussed, we may assume that $K$ is hyperbolic.  By Theorem \ref{thm:genusone}, we may assume $g = 2$.  As we noted in the introduction, the two reducing slopes are consecutive integers.  Hence after possibly mirroring $K$, by (\ref{genusbound}) we assume that these slopes are $2$ and $3$.  Since reducible surgeries on knots in $S^3$ produce non-trivial lens space connected summands, we see that each surgered manifold is the connected sum of a lens space and a homology sphere.  By Lemma~\ref{lem:tperiodic}, we see that 
\begin{equation}\label{eq:2surg}
\dim\HFhat(S^3_2(K),[0]) =\dim\HFhat(S^3_2(K),[1]), 
\end{equation}
and
\begin{equation}\label{eq:3surg}
\dim\HFhat(S^3_3(K),[-1]) =\dim\HFhat(S^3_3(K),[0])=\dim\HFhat(S^3_3(K),[1]). 
\end{equation}

We apply the mapping cone formula to both of these surgeries, simultaneously.  The truncated mapping cone diagram for $2$-surgery is:
$$\xymatrix @=15.pt{
{\Ahat_{-1}} \ar[ddrr]|{h_{-1}}& {\Ahat_{0}}& {\Ahat_{1}} \ar[dd]|{v_{1}} \\ \\
& &{\mathbb{F}}
}$$
Here $\HFhat(S^3_2(K),[0])\cong \Ahat_0$ and $\HFhat(S^3_2(K),[1]) \cong H_*(\text{Cone}(\theta))$, where $\theta = h_{-1} + v_1$.  The truncated mapping cone diagram for $3$-surgery consists only of the complexes $\Ahat_{-1}$, $\Ahat_0$, and $\Ahat_1$, where $\HFhat(S^3_3(K),[i])\cong \Ahat_i$, for $i=-1,0,1$.  Then we see from Proposition \ref{prop:truncatedmcf} and \eqref{eq:2surg} that 
$$\dim \Ahat_{-1}  + \dim \Ahat_{1} + 1 - 2 \rank \theta  = \dim \Ahat_{0}, $$  
while from Proposition \ref{prop:truncatedmcf} and \eqref{eq:3surg} we find that 
$$\dim \Ahat_{-1} =\dim \Ahat_{0} =\dim \Ahat_{1}.$$  
Hence we determine that 
$$\dim \Ahat_{0} + 1 = 2 \rank \theta. $$
Since $\rank \theta $ is either $0$ or $1$, we find that $\dim \Ahat_{0} = 1$.  Thus $S^3_3(K)$ is an $L$-space.  By Theorem~\ref{thm:maintheorem}, $S^3_2(K)$ is irreducible, a contradiction.
\end{proof}

